\RequirePackage{fix-cm}
\documentclass[smallextended]{svjour3}       
\smartqed  
\usepackage{graphicx}
\usepackage{amsmath}
\usepackage{amssymb,latexsym}
\begin{document}

\title{REGULARITY OF BOUNDED TRI-LINEAR MAPS AND THE FOURTH  ADJOINT OF A TRI-DERIVATION
}
\subtitle{Do you have a subtitle?\\ If so, write it here}

\titlerunning{REGULARITY  AND THE FOURTH  ADJOINT }        

\author{Abotaleb Sheikhali \and Ali Ebadian\and Kazem Haghnejad Azar 
}

\authorrunning{A. Sheikhali A. Ebadian K. Haghnejad Azar} 

\institute{A. Sheikhali \at
                        Department of Mathematics, Payame Noor University (PNU), Tehran, Iran.\\
                       \email{Abotaleb.sheikhali.20@gmail.com}           
                                   \and
           A. Ebadian \at
Department of Mathematics, Payame Noor University (PNU), Tehran, Iran.\\           
              \email{Ebadian.ali@gmail.com}
               \and
 K. Haghnejad Azar \at               
Department of Mathematics, University of Mohaghegh Ardabili \at               
              \email{Haghnejad@uma.ac.ir}
}

\date{Received: date / Accepted: date}

\maketitle

\begin{abstract}
In this Article, we give a simple criterion for the regularity of a tri-linear mapping. We provide if $f:X\times Y\times Z\longrightarrow W $ is a bounded tri-linear mapping and $h:W\longrightarrow S$ is a bounded linear mapping, then $f$ is regular if and only if $hof$ is regular. We also  shall give some necessary and sufficient conditions such that the fourth adjoint $D^{****}$ of a tri-derivation $D$ is again tri-derivation.
\end{abstract}

\section{\textbf{Introduction and preliminaries}}

Richard Arens showed in \cite{arens}  that a bounded bilinear map $m:X\times Y\longrightarrow Z $ on normed spaces, has two natural  different extensions $m^{***}$, $m^{r***r}$ from $X^{**}\times Y^{**}$ into $Z^{**}$. When these extensions are equal, $m$ is called Arens regular. A Banach algebra $A$ is said to be Arens regular, if its product $\pi(a,b)=ab$ considered as a bilinear mapping $\pi : A\times A\longrightarrow A$ is Arens regular. The first and second Arens products of $A^{**}$ by symbols $\square$ and $\lozenge$ respectively  defined by 
$$a^{**}\square b^{**}=\pi^{***}(a^{**},b^{**})\ \ \ ,\ \ \ a^{**}\lozenge b^{**}=\pi^{r***r}(a^{**},b^{**}). $$
Some characterizations for the Arens regularity of bounded bilinear map $m$ and Banach algebra $A$ are proved in \cite{1}, \cite{2}, \cite{arens}, \cite{4}, \cite{5}, \cite{11}, \cite{14}, \cite{16} and \cite{17}.
 Suppose $X, Y, Z, W$ and $S$ are normed spaces and $f:X\times Y\times Z\longrightarrow W $ is a bounded tri-linear mapping. In this paper we first define regularity of $f$ map and showing that $f$  is regular if and only if $f^{***r*}(X^{**},W^{*},Z)\subseteq Y^{*}$ and $f^{*****}(W^{*},X^{**},Y^{**})\subseteq Z^{*}$. Also we show that for a bounded tri-linear map $f:X\times Y\times Z\longrightarrow W $ and a bounded linear operator $h:W\longrightarrow S$,  $f$ is regular if and only if $hof$ is regular. 
 
 The natural extensions  of $f$ are  as follows:

\begin{enumerate}
\item  $f^{*}:W^{*}\times X\times Y\longrightarrow Z^{*}$, given by $\langle f^{*}(w^{*},x,y),z\rangle=\langle w^{*},f(x,y,z)\rangle$ where $x\in X,~ y\in Y, ~z\in Z,~ w^{*}\in W^{*}$ ($f^{*}$ is said  the adjoint of $f$ and is a bounded tri-linear map).

\item  $f^{**}=(f^*)^*:Z^{**}\times W^{*}\times X\longrightarrow Y^{*}$, given by $\langle f^{**}(z^{**},~w^{*},x),y\rangle=\langle z^{**}, f^{*}(w^{*},x,y)$ where $x\in X,~ y\in Y,~ z^{**}\in Z^{**}, ~w^{*}\in W^{*}$.

\item $f^{***}=(f^{**})^*:Y^{**}\times Z^{**}\times W^{*}\longrightarrow X^{*}$, given by $\langle f^{***}(y^{**},z^{**},w^{*}),x\rangle=\langle y^{**},  f^{**}(z^{**},w^{*},x) \rangle$ where $x\in X,~ y^{**}\in Y^{**},~ z^{**}\in Z^{**}, ~w^{*}\in W^{*}$.

\item $f^{****}=(f^{***})^*:X^{**}\times Y^{**}\times Z^{**}\longrightarrow W^{**}$, given by $\langle f^{****}(x^{**}, y^{**},z^{**})$\\$,w^{*}\rangle=\langle x^{**}, f^{***}(y^{**},z^{**},w^{*}) \rangle$ where $x^{**}\in X^{**},~ y^{**}\in Y^{**},~ z^{**}\in Z^{**},~ w^{*}\in W^{*}$.
\end{enumerate}
Now let $f^r:Z\times Y\times X\longrightarrow W$ be the flip of $f$ defined by $f^r(z,y,x)=f(x,y,z)$, for every $x\in X ,y\in Y$ and $z\in Z$. Then $f^r$ is a bounded tri-linear map and  it may  extends as above to $f^{r****}:Z^{**}\times Y^{**}\times X^{**}\longrightarrow W^{**}$. When $f^{****}$ and $f^{r****r}$ are equal, then $f$ is said to be regular.

Suppose $A$  is a Banach algebra and $\pi_{1}:A\times X \longrightarrow X$ is a bounded bilinear map. The pair $(\pi_{1},X)$ is said to be a left Banach $A-$module when $\pi_{1}(\pi_1(a,b),x)=\pi_{1}(a,\pi_{1}(b,x))$, for each $a,b\in A$ and $x\in X$.  A right Banach $A-$module may is defined similarly. Let  $\pi_{2}:X\times A \longrightarrow X$ be a bounded bilinear map. The pair $(X,\pi_{2})$ is said to be a right Banach $A-$module if $\pi_{2}(x,\pi_2(a,b))=\pi_{2}(\pi_{2}(x,a),b)$. A triple $(\pi_{1},X,\pi_{2})$ is said to be a Banach $A-$module if  $(X,\pi_{1})$ and $(X,\pi_{2})$ are left and right Banach $A-$modules, respectively, and $\pi_{1}(a,\pi_{2}(x,b))=\pi_{2}(\pi_1(a,x),b)$.
Let $(\pi_{1},X,\pi_{2})$ be a Banach $A-$module. Then $(\pi_{2}^{r*r},X^{*},\pi_{1}^{*})$ is the dual Banach $A-$module of $(\pi_{1},X,\pi_{2})$.

 A bounded linear mapping $D_{1}:A\longrightarrow X^{*}$ is said to be a derivation if for each $a,b\in A$ $$D_{1}(\pi(a,b))=\pi_{1}^{*}(D_{1}(a),b)+\pi_{2}^{r*r}(a,D_{1}(b)).$$ A bounded bilinear map $D_{2}:A\times A\longrightarrow X$(or $X^{*}$)  is called a bi-derivation, if for each $a,b,c$ and $d\in A$
 \begin{eqnarray*}
&&D_{2}(\pi(a,b),c)=\pi_{1}(a,D_{2}(b,c))+\pi_{2}(D_{2}(a,c),b),\\
&&D_{2}(a,\pi(b,c))=\pi_{1}(b,D_{2}(a,c))+\pi_{2}(D_{2}(a,b),c).  
 \end{eqnarray*}
Let $D_{1}:A\longrightarrow A^{*}$ be a derivation. Dales, Rodriguez and Velasco,  in \cite{8} showed that $D_{1}^{**}:(A^{**},\square)\longrightarrow A^{***}$ is a derivation if and only if $\pi^{r****}(D_{1}^{**}(A^{**}),A^{**})$\\$\subseteq A^{*}$.
In \cite{15}, S. Mohamadzadeh and H. Vishki extends this and showed that second adjont $D_{1}^{**}:(A^{**},\square)\longrightarrow A^{***}$ is a derivation if and only if $\pi_{2}^{****}(D_{1}^{**}(A^{**}),X^{**})\subseteq A^{*}$
and which $D_{1}^{**}:(A^{**},\lozenge)\longrightarrow A^{***}$ is a derivation if and only if $\pi_{1}^{r****}(D_{1}^{**}(A^{**}),X^{**})\subseteq A^{*}$.

A. Erfanian Attar et al, provide condition such that the third adjoint $D_{2}^{***}$ of a bi-derivation $D_{2}:A\times A\longrightarrow X$ (or $X^{*}$) is again a bi-derivation, see \cite{10}.  For a Banach $A-$module $(\pi_{1},X,\pi_2)$, the fourth adjoint $D^{****}$ of a tri-derivation $D:A\times A\times A\longrightarrow  X^{*}$  is trivially a tri-linear extension of $D$. A problem which is of interest is under what conditions we need that $D^{****}$ is again a tri-derivation. In section 4  we will extend above mentioned result. 
A bounded trilinear mapping $f:X\times Y\times Z\longrightarrow W $ is said to factor if it is surjective, that is $ f(X\times Y\times Z) =W$.

Throughout the article, we usually identify a normed space with its canonical image in its second dual.

\section{\textbf{Regularity of bounded tri-linear maps}}
\begin{theorem}\label{2.1}
Let $f:X\times Y\times Z\longrightarrow W$ be a bounded tri-linear map. Then $f$ is regular  if and only if 
$$w^{*}-\lim\limits_\alpha w^{*}-\lim\limits_\beta w^{*}-\lim\limits_\gamma f(x_\alpha,y_\beta,z_{\gamma})=w^{*}-\lim\limits_\gamma w^{*}-\lim\limits_\beta w^{*}-\lim\limits_\alpha f(x_\alpha,y_\beta,z_{\gamma}),$$
where $\{x_{\alpha} \}, \{y_{\beta} \}$ and $\{z_{\gamma} \}$ are nets in $X, Y$ and $Z$  which converge to $x^{**}\in X^{**},y^{**}\in Y^{**}$ and $z^{**}\in Z^{**}$  in the $w^{*}-$topologies, respectively.
\end{theorem}
\begin{proof}
For every $w^{*}\in W^{*}$ we have
\begin{eqnarray*}
&&\langle f^{****}(x^{**},y^{**},z^{**}),w^{*}\rangle = \langle x^{**},f^{***}(y^{**},z^{**},w^{*})\rangle\\
&&= \lim\limits_\alpha\langle f^{***}(y^{**},z^{**},w^{*}),x_{\alpha}\rangle= \lim\limits_\alpha\langle y^{**},f^{**}(z^{**},w^{*},x_{\alpha})\rangle\\
&&=\lim\limits_\alpha\lim\limits_\beta\langle f^{**}(z^{**},w^{*},x_{\alpha}),y_{\beta}\rangle=\lim\limits_\alpha\lim\limits_\beta\langle z^{**},f^{*}(w^{*},x_{\alpha},y_{\beta})\rangle\\
&&=\lim\limits_\alpha\lim\limits_\beta\lim\limits_\gamma\langle f^{*}(w^{*},x_{\alpha},y_{\beta}),z_{\gamma}\rangle= \lim\limits_\alpha\lim\limits_\beta\lim\limits_\gamma\langle f(x_{\alpha},y_{\beta},z_{\gamma}),w^{*}\rangle.
\end{eqnarray*}
Therefore $f^{****}(x^{**},y^{**},z^{**})=w^{*}-\lim\limits_\alpha w^{*}-\lim\limits_\beta w^{*}-\lim\limits_\gamma f(x_\alpha,y_\beta,z_{\gamma})$. In the other hands 
 $f^{r****r}(x^{**},y^{**},z^{**})=w^{*}-\lim\limits_\gamma w^{*}-\lim\limits_\beta w^{*}-\lim\limits_\alpha f(x_\alpha,y_\beta,z_{\gamma})$, and proof follows.
\end{proof}
In the following theorem, we provide a criterion concerning to the regularity of a bounded tri-linear map.
\begin{theorem}\label{2.2}
For a bounded tri-linear map $f:X\times Y\times Z\longrightarrow W$  the following statements are equivalent:
\begin{enumerate}
\item $f$ is regular.

\item $f^{*****}=f^{r*******r}.$

\item $f^{***r*}(X^{**},W^{*},Z)\subseteq Y^{*}$ and $f^{*****}(W^{*},X^{**},Y^{**})\subseteq Z^{*}.$
\end{enumerate}
\end{theorem}
\begin{proof}
(1) $\Rightarrow$ (2), if $f$ is regular, then $f^{****}=f^{r****r}$. For every $x^{**}\in X^{**},y^{**}\in Y^{**},z^{**}\in Z^{**}$ and $w^{***}\in W^{***}$ we have 
\begin{eqnarray*}
&&\langle f^{*****}(w^{***},x^{**},y^{**}),z^{**}\rangle =\langle w^{***},f^{****}(x^{**},y^{**},z^{**})\rangle\\
&&=\langle w^{***},f^{r****r}(x^{**},y^{**},z^{**})\rangle=\langle f^{r*******r}(w^{***},x^{**},y^{**}),z^{**}\rangle.
\end{eqnarray*}
as claimed.

(2) $\Rightarrow$ (1), let $f^{*****}=f^{r*******r}$, then for every $w^{*}\in W^{*}$, 
\begin{eqnarray*}
&&\langle f^{r****r}(x^{**},y^{**},z^{**}),w^{*}\rangle=\langle f^{r*******r}(w^{*},x^{**},y^{**}),z^{**}\rangle\\
&&=\langle f^{*****}(w^{*},x^{**},y^{**}),z^{**}\rangle=\langle f^{****}(x^{**},y^{**},z^{**}),w^{*}\rangle.
\end{eqnarray*}
It follows that $f$ is regular.

(1) $\Rightarrow$ (3), assume that $f$ is regular and  $x^{**}\in X^{**},y^{**}\in Y^{**},z\in Z,w^{*}\in W^{*}$. Then we have
\begin{eqnarray*}
&&\langle f^{***r*}(x^{**},w^{*}, z),y^{**}\rangle =\langle f^{****}(x^{**},y^{**},z),w^{*}\rangle\\
&&=\langle f^{r****r}(x^{**},y^{**},z),w^{*}\rangle=\langle f^{r**}(x^{**},w^{*},z),y^{**}\rangle.
\end{eqnarray*}
Therefore  $f^{***r*}(x^{**},w^{*},z)=f^{r**}(x^{**},w^{*},z)\in Y^{*}$. So $f^{***r*}(X^{**},W^{*},Z)\subseteq Y^{*}$. A similar argument shows that $f^{*****}(w^{*},x^{**},y^{**})=f^{r***r}(w^{*},x^{**},y^{**})\in Z^{*}$. Thus $f^{*****}(W^{*},X^{**},Y^{**})\subseteq Z^{*}$, as claimed.

(3) $\Rightarrow$ (1), let $\{x_{\alpha} \}, \{y_{\beta} \}$ and $\{z_{\gamma} \}$ are nets in $X, Y$ and $Z$  which converge to $x^{**},y^{**}$ and $z^{**}$  in the $w^{*}-$topologies, respectively. For every $w^{*}\in W^{*}$ we have
\begin{eqnarray*}
&&\langle f^{r****r}(x^{**},y^{**},z^{**}),w^{*}\rangle =\lim\limits_\gamma\lim\limits_\beta\lim\limits_\alpha\langle f(x_{\alpha},y_{\beta},z_{\gamma}),w^{*}\rangle\\
&&=\lim\limits_\gamma\lim\limits_\beta\lim\limits_\alpha\langle f^{***}(y_{\beta},z_{\gamma},w^{*}),x_{\alpha}\rangle=\lim\limits_\gamma\lim\limits_\beta\langle x^{**},f^{***}(y_{\beta},z_{\gamma},w^{*})\\
&&=\lim\limits_\gamma\lim\limits_\beta\langle x^{**},f^{***r}(w^{*},z_{\gamma},y_{\beta})\rangle=\lim\limits_\gamma\lim\limits_\beta\langle f^{***r*}(x^{**},w^{*},z_{\gamma}),y_{\beta}\rangle\\
&&=\lim\limits_\gamma \langle f^{***r*}(x^{**},w^{*},z_{\gamma}),y^{**}\rangle=\lim\limits_\gamma \langle x^{**}, f^{***r}(w^{*},z_{\gamma},y^{**})\rangle\\
&&=\lim\limits_\gamma \langle x^{**}, f^{***}(y^{**},z_{\gamma},w^{*})\rangle=\lim\limits_\gamma \langle f^{****}(x^{**},y^{**},z_{\gamma}),w^{*}\rangle\\
&&=\lim\limits_\gamma \langle f^{*****}(w^{*},x^{**},y^{**}),z_{\gamma}\rangle=f^{*****}(w^{*},x^{**},y^{**}),z^{**}\rangle\\
&&=\langle f^{****}(x^{**},y^{**},z^{**}),w^{*}\rangle.
\end{eqnarray*}
It follows that $f$ is regular and this completes the proof.
\end{proof}

\begin{corollary}\label{2.3}
For a bounded tri-linear map $f:X\times Y\times Z\longrightarrow W$  the following statements are equivalent:
\begin{enumerate}
\item $f$ is regular.

\item $f^{r*****r}=f^{*******}$.

\item $f^{r***r*}(Z^{**},W^{*},X)\subseteq Y^{*}$ and $f^{*****}(W^{*},Z^{**},Y^{**})\subseteq X^{*}.$
\end{enumerate}
\end{corollary}
\begin{proof}
The mapping $f$ is regular if and only if $f^{r}$ is regular. Therefore by Theorem \ref{2.2}, the desired result is obtained.
\end{proof}
\begin{corollary}\label{2.4}
For a bounded tri-linear map $f:X\times Y\times Z\longrightarrow W$, if from $X, Y$ or $Z$  at least two reflexive then $f$ is regular.
\end{corollary}
\begin{proof}
Without having to enter the whole argument, let $Y$ and $Z$ are reflexive. Since $Y$ is reflexive,  $Y^{*}=Y^{***}$. Therefore
$$f^{***r*}(X^{**},W^{*},Z^{**})\subseteq Y^{***}=Y^{*} \ \ \ \ \ \ \ \ \ \ \ \ \ \ \ (2-1)$$
In the other hands, since $Z$ is the reflexive space, thus 
$$f^{*****}(W^{***},X^{**},Y^{**})\subseteq Z^{***}=Z^{*} \ \ \ \ \ \ \ \ \ \ \ \ \ \ \ (2-2)$$
 Now Using (2-1), (2-2) and Theorem \ref{2.2}, the result holds.
\end{proof}
\begin{corollary}\label{2.5}
Let bounded tri-linear map $f:X\times Y\times Z\longrightarrow W$ be regular. Then
\begin{enumerate}
\item If $f^{***r*}(X^{**},W^{*},Z)$  factors, then $Y$  is reflexive space.

\item If $f^{*****}(W^{*},X^{**},Y^{**})$ factors, then $Z$  is reflexive space.

\item If $f^{****r*}(W^{*},Z,Y)$ factors, then $X$  is reflexive space.
\end{enumerate}
\end{corollary}
\begin{proof}
(1) Let $f$ be regular. It follows that  $f^{***r*}(X^{**},W^{*},Z)\subseteq Y^{*}$. In the other hands, $f^{***r*}(X^{**},W^{*},Z)$ is factor. So for each $y^{***}\in Y^{***}$ there exist $x^{**}\in X^{**},w^{*}\in W^{*}$ and $z\in Z$ such that $f^{***r*}(x^{**},w^{*},z)=y^{***}$. Therefore $Y^{***}\subseteq Y^{*}$.

(2) The proof  similar to (1).

(3) Enough show that $f^{****r*}(W^{*},Z,Y)\subseteq X^{*}$ whenever  $f$ is regular. For every $x^{**}\in X^{**}, y\in Y, z\in Z$ and $w^{*}\in W^{*}$ we have 
\begin{eqnarray*}
&&\langle f^{****r*}(w^{*},z,y),x^{**}\rangle=\langle w^{*},f^{****}(x^{**},y,z)\rangle\\
&&=\langle f^{r****r}(x^{**},y,z),w^{*}\rangle=\langle f^{r*}(w^{*},z,y),x^{**}\rangle.
\end{eqnarray*}
Therefore $f^{****r*}(w^{*},z,y)=f^{r*}(w^{*},z,y)\in X^{*}$. The rest of proof has similar argument such as (1).
\end{proof}

\begin{corollary}\label{2.6}
If $I_{X}$, $I_{Y}$ and $I_{Z}$  are weakly compact identity mapping, then all of them and all of their adjoints  are regular.
\end{corollary}

\begin{example}\label{2.7}
\begin{enumerate}
\item Let $G$ be a compact group. Let $1 < p,q < \infty$ and $\frac{1}{p}+\frac{1}{q}=1+\frac{1}{r}$. Then by \cite[Sections 2.4 and 2.5]{12}, we conclude that $L^{1}(G)\star L^{p}(G)\subset L^{p}(G)$ and $L^{p}(G)\star L^{q}(G)\subset L^{r}(G)$ where $(g\star h)(x)=\int_{G}g(y)h(y^{-1}x)dy$ for $x\in G$.  Since the Banach spaces $L^{p}(G)$ and $L^{q}(G)$ are reflexive, thus by corollary \ref{2.4} we conclude that the bounded tri-linear mapping 
$$f:L^{1}(G)\times L^{p}(G)\times L^{q}(G)\longrightarrow L^{r}(G)$$
 defined by $f(k,g,h)=(k\star g)\star h$, is regular for every $k\in L^{1}(G), g\in L^{p}(G)$ and $h\in L^{q}(G)$.  
\item Let $G$ be a locally compact group. We know from \cite{18} that $L^{1}(G)$ is regular if and only if it is reflexive or  $G$ is finite. It follows that for every finite locally compact group $G$, by corollary \ref{2.4},  the bounded tri-linear mapping $f:L^{1}(G)\times L^{1}(G)\times L^{1}(G)\longrightarrow L^{1}(G)$ defined by $f(k,g,h)=k\star g\star h$, is regular for every $k, g$ and $h\in L^{1}(G)$.  
\item $C^{*}-$algebras are standard examples of Banach algebras that are Arens regular, see\cite{6}. We know that a $C^{*}-$algebra is reflexive if and only if it is of finite dimension. Since if $A$ is a finite dimension $C^{*}$-algebra, then  by  corollary \ref{2.4},  we conclude that the bounded tri-linear mapping $f:A\times A\times A\longrightarrow A$ is regular.
\item Let $G$ be a locally compact group and let $M(G)$ be measure algebra of $G$, see \cite[Section 2.5]{12}.  Let the convolution for $\mu_{1},\mu_{2}\in M(G)$ defined by 
$$\int \psi d(\mu_{1}*\mu_{2})=\int\int \psi(xy)d\mu_{1}(x)d\mu_{2}(y),\ \ (\psi\in C_{0}(G)).$$
We have
\begin{eqnarray*}
\int \psi d(\mu_{1}*(\mu_{2}*\mu_{3}))&=&\int\int\int \psi(xyz)d\mu_{1}(x)d\mu_{2}(y)d\mu_{3}(z)\\
&=&\int \psi d((\mu_{1}*\mu_{2})*\mu_{3})
\end{eqnarray*}
 for $\mu_{1},\mu_{2}$ and $\mu_{3}\in M(G)$. Therefore convolution is associative. Now we define the bounded tri-linear mapping $$f:M(G)\times M(G)\times M(G)\longrightarrow M(G)$$
by $f(\mu_{1},\mu_{2},\mu_{3})=\int \psi d(\mu_{1}*\mu_{2}*\mu_{3})$. If $G$ is finite,  then $f$ is regular.
\end{enumerate}
\end{example}

\section{\textbf{Some results for regularity}}
Dales, Rodriguez-Palacios and Velasco in  \cite[Theorem 4.1]{8},  for a bonded bilinear map $m:X\times Y\longrightarrow Z$ have shown that, $m^{r*r***}=m^{***r*r}$ if and only if both $m$ and $m^{r*}$ are Arens regular. Now in the following we study it in general case.

\begin{remark}\label{3.1}
In the next theorem, $f^{n}$ is $n-$th adjoint of $f$ for each $n\in N$. 
\end{remark}

\begin{theorem}\label{3.2}
If $f$ and $f^{rn}$ are reular, then $f^{4rnr}=f^{rnr4}$.
\end{theorem}
\begin{proof}
Since $f$ is regular, so $f^{4r}=f^{r4}$. Therefore $f^{4rn}=f^{r(n+4)}$. In the other hands, regularity of $f^{rn}$  follows that $f^{r(n+4)}=f^{rnr4r}$. Thus $f^{rnr4r}=f^{4rn}$ and this completes the proof.
\end{proof}

\begin{theorem}\label{3.3}
Let $f:X\times Y\times Z\longrightarrow W$ be a bounded tri-linear mapping. Then
\begin{enumerate}
\item $f^{****r**r}=f^{r**r****}$ if and only if both $f$ and $f^{r**}$  are regular.
\item $f^{****r***r}=f^{r***r****}$ if and only if both $f$ and $f^{r***}$  are regular.
\end{enumerate}
\end{theorem}
\begin{proof}
We prove only (1), the other part has the same argument. If  both $f$ and $f^{r**}$  are regular, then by applying Theorem \ref{3.2}, for $n=2$, $f^{****r**r}=f^{r**r****}$.

 Conversely, suppose that $f^{****r**r}=f^{r**r****}$. First we show that $f$ is regular. Let  $\{z_{\gamma} \}$ is net in $Z$  which converge to $z^{**}\in Z^{**}$  in the $w^{*}-$topologies. Then for every $x^{**}\in X^{**}, y^{**}\in Y^{**}$ and $w^{*}\in W^{*}$ we have
 \begin{eqnarray*}
&&\langle f^{****}(x^{**},y^{**},z^{**}),w^{*}\rangle=\langle f^{****r}(z^{**},y^{**},x^{**}),w^{*}\rangle\\
&&=\langle f^{****r**r}(z^{**},w^{*},x^{**}),y^{**}\rangle=\langle f^{r**r****}(z^{**},w^{*},x^{**}),y^{**}\rangle\\
&&=\lim\limits_\gamma\langle y^{**},f^{r**r}(z_{\gamma},w^{*},x^{**})\rangle=\langle f^{r****r}(x^{**},y^{**},z^{**}),w^{*}\rangle.
\end{eqnarray*}
Therefore $f$ is regular. Now we show that $f^{r**}$ is regular. Let $\{x_{\alpha}^{**} \}$ be net in $X^{**}$  which converge to $x^{****}\in X^{****}$  in the $w^{*}-$topologies. Then for every $y^{**}\in Y^{**}, z^{**}\in Z^{**}$ and $w^{***}\in W^{***}$ we have
 \begin{eqnarray*}
&&\langle f^{r**r****r}(x^{****},w^{***},z^{**}),y^{**}\rangle =\langle f^{r**r****}(z^{**},w^{***},x^{****}),y^{**}\rangle\\
&&=\langle f^{****r**r}(z^{**},w^{***},x^{****}),y^{**}\rangle=\lim\limits_\alpha\langle w^{***},f^{****}(x_{\alpha}^{**},y^{**},z^{**})\rangle\\
&&=\lim\limits_\alpha\langle w^{***},f^{r****r}(x_{\alpha}^{**},y^{**},z^{**})\rangle=\lim\limits_\alpha\langle w^{***},f^{r****}(z^{**},y^{**},x_{\alpha}^{**})\rangle\\
&&=\langle  f^{r******}(x^{****},w^{***},z^{**}),y^{**}\rangle.
\end{eqnarray*}
It follows that $f^{r**}$ is regular and this completes the proof.
\end{proof}

Arens has shown \cite{arens} that a bounded bilinear map $m$ is regular if and only if for each $z^{*}\in Z^{*}$, the bilinear form $z^{*}om$ is regular. In the next theorem we give an important characterization of  regularity bounded tri-linear mappings.
\begin{lemma}\label{3.4}
Suppose $X, ~Y,~ Z, ~W$ and $S$  are normed spaces and $f:X\times Y\times Z\longrightarrow W $  and $h:W\longrightarrow S$ are bounded tri-linear mapping and bounded linear mapping, respectively. Then we have
\begin{enumerate}
\item $h^{**}of^{****}=(hof)^{****}$.
\item $h^{**}of^{r****r}=(hof)^{r****r}$.
\end{enumerate}
\end{lemma}
\begin{proof}
Let $\{x_{\alpha} \}, \{y_{\beta} \}$ and $\{z_{\gamma} \}$ be nets in $X, Y$ and $Z$  which converge to $x^{**}\in X^{**},y^{**}\in Y^{**}$ and $z^{**}\in Z^{**}$  in the $w^{*}-$topologies, respectively. For each $s^{*}\in S^{*}$ we have
\begin{eqnarray*}
&&\langle h^{**}of^{****}(x^{**},y^{**},z^{**}),s^{*}\rangle=\langle h^{**}(f^{****}(x^{**},y^{**},z^{**})),s^{*}\rangle\\
 &&=\langle f^{****}(x^{**},y^{**},z^{**}),h^{*}(s^{*})\rangle=\lim\limits_\alpha \lim\limits_\beta \lim\limits_\gamma\langle
 h^{*}(s^{*}),f(x_{\alpha},y_{\beta},z_{\gamma})\rangle\\
&&=\lim\limits_\alpha \lim\limits_\beta \lim\limits_\gamma\langle s^{*},h(f(x_{\alpha},y_{\beta},z_{\gamma})) \rangle=\lim\limits_\alpha \lim\limits_\beta \lim\limits_\gamma\langle s^{*},hof(x_{\alpha},y_{\beta},z_{\gamma})\rangle\\
&&=\langle (hof)^{****}(x^{**},y^{**},z^{**}),s^{*}\rangle.
\end{eqnarray*}
Hence $h^{**}of^{****}(x^{**},y^{**},z^{**})=(hof)^{****}(x^{**},y^{**},z^{**})$. A similar argument applies for (2).
\end{proof}
\begin{theorem}\label{3.5}
Let $f:X\times Y\times Z\longrightarrow W $ and $h:W\longrightarrow S$ be bounded tri-linear mapping and bounded linear mapping, respectively. Then $f$ is regular if and only if $hof$ is regular.
\end{theorem}
\begin{proof}
Assume that $f$ is regular. Then for every $x^{**}\in X^{**},y^{**}\in Y^{**},z^{**}\in Z^{**}$ and $s^{*}\in S^{*}$ we have
\begin{eqnarray*}
&&\langle h^{**}(f^{r****r}(x^{**},y^{**},z^{**})),s^{*}\rangle=\langle f^{r****r}(x^{**},y^{**},z^{**}),h^{*}(s^{*})\rangle\\
&&=\langle f^{****}(x^{**},y^{**},z^{**}),h^{*}(s^{*})\rangle=\langle h^{**}(f^{****}(x^{**},y^{**},z^{**})),s^{*}\rangle.
\end{eqnarray*}
Therefore $h^{**}of^{r****r}(x^{**},y^{**},z^{**})=h^{**}of^{****}(x^{**},y^{**},z^{**})$ and by applying Lemma \ref{3.4}, we implies that
 $$(hof)^{r****r}(x^{**},y^{**},z^{**})=(hof)^{****}(x^{**},y^{**},z^{**}).$$
It follows that $hof$ is regular.  

For the converse, suppose that $hof$ is regular. By contradiction, let $f$ be not regular. Thus there  exist  $x^{**}\in X^{**},y^{**}\in Y^{**}$ and $z^{**}\in Z^{**}$  such  that $f^{****}(x^{**},y^{**},z^{**})\neq f^{r****r}(x^{**},y^{**},z^{**})$. Therefore we have 
\begin{eqnarray*}
&&(hof)^{****}(x^{**},y^{**},z^{**})=w^{*}-\lim\limits_\alpha w^{*}-\lim\limits_\beta w^{*}-\lim\limits_\gamma (hof)(x_{\alpha},y_{\beta},z_{\gamma})\\
&&=\lim\limits_\alpha \lim\limits_\beta \lim\limits_\gamma\langle f(x_{\alpha},y_{\beta},z_{\gamma}),h\rangle=\langle f^{****}(x^{**},y^{**},z^{**}),h\rangle\\
&&\neq  \langle f^{r****r}(x^{**},y^{**},z^{**}),h\rangle=\lim\limits_\gamma \lim\limits_\beta \lim\limits_\alpha\langle f(x_{\alpha},y_{\beta},z_{\gamma}),h\rangle\\
&&=w^{*}-\lim\limits_\gamma w^{*}-\lim\limits_\beta w^{*}-\lim\limits_\alpha (hof)(x_{\alpha},y_{\beta},z_{\gamma})\\
&&=(hof)^{r****r}(x^{**},y^{**},z^{**}).
\end{eqnarray*}
It follows that  $(hof)^{****}(x^{**},y^{**},z^{**})\neq (hof)^{r****r}(x^{**},y^{**},z^{**})$.  
\end{proof}
Another interesting case of regularity is in the following.
\begin{theorem}\label{3.6}
Let $f:X\times Y\times Z\longrightarrow W $ be a bounded tri-linear mapping, $m:X\times Y\longrightarrow Z$ be a bounded bilinear mapping, $T:X\times Y\longrightarrow W$ defined by $T(x,y)=f(x,y,m(x,y))$  for each $x\in X, y\in Y$ and $m^{***}$ is factors. Then $T$ is regular if and only if $f$ is regular.
\end{theorem}
\begin{proof}
Let  $T$ be regular. Since the mapping $m^{***}:X^{**}\times Y^{**}\longrightarrow Z^{**}$ is onto, so for each $z^{**}\in Z^{**}$ there  exist  $x^{**}\in X^{**}$ and $y^{**}\in Y^{**}$ such  that $m^{***}(x^{**},y^{**})=z^{**}$. Suppose that $\{x_{\alpha} \}$ and $\{y_{\beta} \}$ are nets in $X$ and $Y$ which converge to $x^{**}$ and $y^{**}$ in the $w^{*}-$topologies, respectively. For every $w^{*}\in W^{*}$ we have 
\begin{eqnarray*}
&&\langle f^{****}(x^{**},y^{**},z^{**}),w^{*}\rangle=\langle  f^{****}(x^{**},y^{**},m^{***}(x^{**},y^{**})),w^{*}\rangle\\
&&=\lim\limits_{\alpha}\lim\limits_{\beta}\lim\limits_{\alpha}\lim\limits_{\beta}\langle w^{*},f(x_{\alpha},y_{\beta},m(x_{\alpha},y_{\beta}))\rangle=\lim\limits_{\alpha}\lim\limits_{\beta}\lim\limits_{\alpha}\lim\limits_{\beta}\langle w^{*},T(x_{\alpha},y_{\beta})\rangle\\
&&=\lim\limits_{\alpha}\lim\limits_{\beta}\langle T^{***}(x^{**},y^{**}),w^{*}\rangle=\lim\limits_{\alpha}\lim\limits_{\beta}\langle T^{r***r}(x^{**},y^{**}),w^{*}\rangle\\
&&=\lim\limits_{\alpha}\lim\limits_{\beta}\lim\limits_{\beta}\lim\limits_{\alpha}\langle w^{*},T(x_{\alpha},y_{\beta})\rangle=\lim\limits_{\alpha}\lim\limits_{\beta}\lim\limits_{\beta}\lim\limits_{\alpha}\langle w^{*},f(x_{\alpha},y_{\beta},m(x_{\alpha},y_{\beta}))\rangle\\
&&=\langle f^{r****r}(x^{**},y^{**},m^{***}(x^{**},y^{**})),w^{*}\rangle=\langle f^{r****r}(x^{**},y^{**},z^{**}),w^{*}\rangle.
\end{eqnarray*}
It follows that $f$ is regular. 

For the converse, suppose that $f$ is regular. For every $w^{*}\in W^{*}$ we have
\begin{eqnarray*}
&&\langle T^{***}(x^{**},y^{**}),w^{*}\rangle=\lim\limits_{\alpha}\lim\limits_{\beta}\langle w^{*},T(x_{\alpha},y_{\beta})\rangle\\
&&=\lim\limits_{\alpha}\lim\limits_{\beta}\langle w^{*},f(x_{\alpha},y_{\beta},m(x_{\alpha},y_{\beta}))\rangle=\langle f^{****}(x^{**},y^{**},m^{***}(x_{\alpha},y_{\beta})),w^{*}\rangle\\
&&=\langle f^{r****r}(x^{**},y^{**},m^{***}(x_{\alpha},y_{\beta})),w^{*}\rangle=\lim\limits_{\beta}\lim\limits_\alpha\langle w^{*},f(x_{\alpha},y_{\beta},m(x_{\alpha},y_{\beta}))\rangle\\
&&=\lim\limits_{\beta}\lim\limits_\alpha\langle w^{*},T(x_{\alpha},y_{\beta})\rangle=\langle T^{r***r}(x^{**},y^{**}),w^{*}\rangle.
\end{eqnarray*}
Therefore $T^{***}=T^{r***r}$ , as claimed.
\end{proof}
\begin{theorem}\label{3.7}
Let $X, Y, Z, W$ and $S$ be Banach spaces, $f:X\times Y\times Z\longrightarrow W $ be a bounded tri-linear mapping and $x\in X,y\in Y, z\in Z$. Then
\begin{enumerate}
\item Let $g_{1}:S\times Y\times Z\longrightarrow W $ be a bounded tri-linear mapping and let $h_{1}:X\longrightarrow S$ be a bounded linear mapping such that $f(x,y,z)=g_{1}(h_{1}(x),y,z)$. If $h_{1}$ is weakly compact, then $f^{****r*}(W^{***},Z^{**},Y^{**})\subseteq X^{*}$.

\item Let $g_{2}:X\times S\times Z\longrightarrow W $ be a bounded tri-linear mapping and let $h_{2}:Y\longrightarrow S$ be a bounded linear mapping such that $f(x,y,z)=g_{2}(x,h_{2}(y),z)$. If $h_{2}$ is weakly compact, then $f^{***r*}(X^{**},W^{*},Z^{**})\subseteq Y^{*}$.

\item Let $g_{3}:X\times Y\times S\longrightarrow W $ be a bounded tri-linear mapping and let $h_{3}:Z\longrightarrow S$ be a bounded linear mapping such that $f(x,y,z)=g_{3}(x,y,h_{3}(z))$. If $h_{3}$ is weakly compact, then $f^{*****}(W^{***},X^{**},Y^{**})\subseteq Z^{*}$.
\end{enumerate}
\end{theorem}
\begin{proof}
We prove only (1), the other parts have the same argument. For every $x\in X,y\in Y, z\in Z$ and $w^{*}\in W^{*}$ we have 
$$\langle f^{*}(w^{*},x,y),z \rangle=\langle w^{*},f(x,y,z)\rangle=\langle w^{*},g_{1}(h_{1}(x),y,z)\rangle=\langle g_{1}^{*}(w^{*},h_{1}(x),y),z\rangle.$$
Therefore $f^{*}(w^{*},x,y)=g_{1}^{*}(w^{*},h_{1}(x),y)$, and implies that  for every $z^{**}\in Z^{**}$,
\begin{eqnarray*}
&&\langle f^{**}(z^{**},w^{*},x),y \rangle=\langle z^{**},f^{*}(w^{*},x,y)\rangle\\
&&=\langle z^{**},g_{1}^{*}(w^{*},h_{1}(x),y)\rangle=\langle g_{1}^{**}(z^{**},w^{*},h_{1}(x)),y\rangle.
\end{eqnarray*}
So $f^{**}(z^{**},w^{*},x)=g_{1}^{**}(z^{**},w^{*},h_{1}(x))$ and implies that  for every $y^{**}\in Y^{**}$,
\begin{eqnarray*}
&&\langle f^{***}(y^{**},z^{**},w^{*}),x \rangle=\langle y^{**}, f^{**}(z^{**},w^{*},x)\rangle=\langle y^{**},g_{1}^{**}(z^{**},w^{*},h_{1}(x))\rangle\\
&&=\langle g_{1}^{***}(y^{**},z^{**},w^{*}),h_{1}(x)\rangle=\langle h_{1}^{*}(g_{1}^{***}(y^{**},z^{**},w^{*})),x\rangle.
\end{eqnarray*}
Thus $f^{***}(y^{**},z^{**},w^{*})=h_{1}^{*}(g_{1}^{***}(y^{**},z^{**},w^{*}))$ and implies that  for every $x^{**}\in X^{**}$,
\begin{eqnarray*}
&&\langle f^{****}(x^{**},y^{**},z^{**}),w^{*} \rangle=\langle x^{**}, f^{***}(y^{**},z^{**},w^{*})\rangle \\ 
&&=\langle x^{**},h_{1}^{*}(g_{1}^{***}(y^{**},z^{**},w^{*}))\rangle=\langle h_{1}^{**}(x^{**}),(g_{1}^{***}(y^{**},z^{**},w^{*})\rangle \\
&&=\langle g_{1}^{****}(h_{1}^{**}(x^{**}),y^{**},z^{**}),w^{*}\rangle.
\end{eqnarray*}
Therefore for every $w^{***}\in W^{***}$ we have 
\begin{eqnarray*}
&&\langle f^{****r*}(w^{***},z^{**},y^{**}),x^{**} \rangle =\langle w^{***}, f^{****r}(z^{**},y^{**},x^{**})\rangle\\
&&=\langle w^{***}, f^{****}(x^{**},y^{**},z^{**})\rangle=\langle w^{***}, g_{1}^{****}(h_{1}^{**}(x^{**}),y^{**},z^{**})\rangle\\
&&=\langle w^{***}, g_{1}^{****r}(z^{**},y^{**},h_{1}^{**}(x^{**}))\rangle=\langle g_{1}^{****r*}(w^{***},z^{**},y^{**}),h_{1}^{**}(x^{**})\rangle \\
&&=\langle h_{1}^{***}(g_{1}^{****r*}(w^{***},z^{**},y^{**})),x^{**}\rangle.
\end{eqnarray*}
Therefore $f^{****r*}(w^{***},z^{**},y^{**})=h_{1}^{***}(g_{1}^{****r*}(w^{***},z^{**},y^{**}))$. The weak compactness of $h_{1}$ implies that of $h_{1}^{*}$, from which we have $h_{1}^{***}(S^{***})\subseteq X^{*}$. Thus $h_{1}^{***}(g_{1}^{****r*}(w^{***},z^{**},y^{**})) \in X^{*}$ and this completes the proof.
\end{proof}
This theorem, combined with Theorem \ref{2.2}, yields the next result.

\begin{corollary}\label{3.8}
With the assumptions Theorem \ref{3.7}, if $h_{2}$ and $h_{3}$ are weakly compact, then $f$ is regular.
\end{corollary}
\begin{proof}
Both $h_{2}$ and $h_{3}$ are weakly compact, so by Theorem \ref{3.7} we have 
$$f^{***r*}(X^{**},W^{*},Z^{**})\subseteq Y^{*}\ \ \ , \ \ \ f^{*****}(W^{***},X^{**},Y^{**})\subseteq Z^{*}.$$
In particular
$$f^{***r*}(X^{**},W^{*},Z)\subseteq Y^{*}\ \ \ , \ \ \ f^{*****}(W^{*},X^{**},Y^{**})\subseteq Z^{*}.$$
Now by Theorem \ref{2.2}, $f$ is regular.
\end{proof}
The  converse of previous result is not true in general sense as following corollary.
\begin{corollary}\label{3.9}
With the assumptions Theorem \ref{3.7}, if $f$ is regular and both $g_{2}^{***r*}$ and $g_{3}^{*****}$ are factors, then $h_{2}$ and $h_{3}$ are weakly compact. 
\end{corollary}
\begin{proof}
Since $f^{***r*}(X^{**},W^{*},Z^{**})=h_{2}^{***}(g_{2}^{***r*}(X^{**},W^{*},Z^{**}))$, so $h_{2}^{***}(g_{2}^{***r*}$\\$(X^{**},W^{*},Z^{**}))\subseteq Y^{*}$. In the other hands $g_{2}^{***r*}$ is factors, so implies that $h_{2}^{***}(S^{***})\subseteq Y^{*}$. Therefore $h_{2}^{*}$ is weakly compact and implies that $h_{2}$ is weakly compact. The other part has the same argument for $h_{3}$.
\end{proof}

\section{\textbf{The fourth adjoint of a tri-derivation}}

\begin{definition}\label{4.1}
Let $(\pi_{1},X,\pi_2)$ be a  Banach $A-$module.  A bounded tri-linear mapping $D:A\times A\times A\longrightarrow X$  is
said to be a tri-derivation when
\begin{enumerate}
\item $D(\pi (a,d),b,c)=\pi_{2}(D(a,b,c),d)+\pi_{1}(a,D(d,b,c))$,

\item $D(a,\pi (b,d),c)=\pi_{2}(D(a,b,c),d)+\pi_{1}(b,D(a,d,c))$,

\item $D(a,b,\pi(c,d))=\pi_{2}(D(a,b,c),d)+\pi_{1}(c,D(a,b,d))$,
\end{enumerate}
for each $a, b,c,d \in A$. If  $(\pi_{1},X,\pi_{2})$ is a Banach $A-$module, then $(\pi_{2}^{r*r},X^{*},\pi_{1}^{*})$ is the dual Banach $A-$module of $(\pi_{1},X,\pi_{2})$. Therefore a bounded tri-linear mapping $D:A\times A\times A\longrightarrow  X^{*}$  is a tri-derivation when
\begin{enumerate}
\item $D(\pi (a,d),b,c)=\pi_{1}^{*}(D(a,b,c),d)+\pi_{2}^{r*r}(a,D(d,b,c))$,

\item $D(a,\pi (b,d),c)=\pi_{1}^{*}(D(a,b,c),d)+\pi_{2}^{r*r}(b,D(a,d,c))$,

\item $D(a,b,\pi(c,d))=\pi_{1}^{*}(D(a,b,c),d)+\pi_{2}^{r*r}(c,D(a,b,d))$.
\end{enumerate}
It can also be written, a bounded tri-linear mapping $D:A\times A\times A\longrightarrow A$  is said to be a tri-derivation when
\begin{enumerate}
\item $D(\pi (a,d),b,c)=\pi(D(a,b,c),d)+\pi(a,D(d,b,c))$,

\item $D(a,\pi (b,d),c)=\pi(D(a,b,c),d)+\pi(b,D(a,d,c))$,

\item $D(a,b,\pi(c,d))=\pi(D(a,b,c),d)+\pi(c,D(a,b,d))$.
\end{enumerate}
\end{definition}
\begin{example}\label{4.2}
Let $A$ be a Banach algebra, for any $a, b\in A$ the symbol $[a,b]=ab-ba$ stands for multiplicative commutator of $a$ and $b$. Let  $M_{n\times n}(C)$ be the Banach algebra of all $n\times n$ matrix and $A=\{ \begin{pmatrix} x & y \\ 0 & 0 \end{pmatrix}\in M_{n\times n}(C)| ~  x, y \in C\}$. Then $A$ is Banach algebra with the norm
\begin{equation*}
\parallel a\parallel=(\Sigma_{i,j}|\alpha_{ij}|^2)^{\frac{1}{2}} \ \ \ \ ,\ (a=(\alpha_{ij})\in A).
\end{equation*}
We define $D:A\times A\times A\longrightarrow A$ to be the bounded tri-linear map given by 
\begin{equation*}
D(a,b,c)=[ \begin{pmatrix} 0 & 1 \\ 0 & 0 \end{pmatrix},abc] \ \ \ ,\ \ \ (a,b,c \in A).
\end{equation*}
Then for $a= \begin{pmatrix}
x_{1} & y_{1} \\ 0 & 0
\end{pmatrix}, b=\begin{pmatrix} 
x_{2} & y_{2} \\ 0 & 0
\end{pmatrix}, c=\begin{pmatrix} 
x_{3} & y_{3} \\ 0 & 0
\end{pmatrix}$ and $d=\begin{pmatrix} x_{4} & y_{4} \\ 0 & 0 \end{pmatrix}\in A$ we have 

$D(\pi (a,d),b,c)= D(\begin{pmatrix} 
x_{1}x_{4} & x_{1}y_{4} \\ 0 & 0
\end{pmatrix},\begin{pmatrix} 
x_{2} & y_{2} \\ 0 & 0
\end{pmatrix},\begin{pmatrix} 
x_{3} & y_{3} \\ 0 & 0
\end{pmatrix}) \\
=[\begin{pmatrix} 
0 & 1 \\ 0 & 0
\end{pmatrix},\begin{pmatrix} 
x_{1}x_{2}x_{3}x_{4} & x_{1}x_{2}x_{4}y_{3} \\ 0 & 0
\end{pmatrix}]= \begin{pmatrix} 
0 & -x_{1}x_{2}x_{3}x_{4} \\ 0 & 0
\end{pmatrix} \\
=\begin{pmatrix} 
0 & -x_{1}x_{2}x_{3} \\ 0 & 0
\end{pmatrix}\begin{pmatrix} 
x_{4} & y_{4} \\ 0 & 0
\end{pmatrix}+\begin{pmatrix} 
x_{1} & y_{1} \\ 0 & 0
\end{pmatrix}\begin{pmatrix} 
0 & -x_{2}x_{3}x_{4} \\ 0 & 0
\end{pmatrix}\\
=(\begin{pmatrix} 
0 & 0 \\ 0 & 0
\end{pmatrix}-\begin{pmatrix} 
0 & x_{1}x_{2}x_{3} \\ 0 & 0
\end{pmatrix})\begin{pmatrix} 
x_{4} & y_{4} \\ 0 & 0
\end{pmatrix}+\begin{pmatrix} 
x_{1} & y_{1} \\ 0 & 0
\end{pmatrix}(\begin{pmatrix} 
0 & 0 \\ 0 & 0
\end{pmatrix}-\begin{pmatrix} 
0 & x_{2}x_{3}x_{4} \\ 0 & 0
\end{pmatrix}) \\
= (\begin{pmatrix} 
0 & 1 \\ 0 & 0
\end{pmatrix}\begin{pmatrix} 
x_{1}x_{2}x_{3} & x_{1}x_{2}y_{3} \\ 0 & 0
\end{pmatrix}-\begin{pmatrix} 
x_{1}x_{2}x_{3} & x_{1}x_{2}y_{3} \\ 0 & 0
\end{pmatrix}\begin{pmatrix} 
0 & 1 \\ 0 & 0
\end{pmatrix})\begin{pmatrix} 
x_{4} & y_{4} \\ 0 & 0
\end{pmatrix} \\
+ \begin{pmatrix} 
x_{1} & y_{1} \\ 0 & 0
\end{pmatrix}(\begin{pmatrix} 
0 & 1 \\ 0 & 0
\end{pmatrix}\begin{pmatrix} 
x_{2}x_{3}x_{4} & x_{2}x_{4}y_{3} \\ 0 & 0
\end{pmatrix}-\begin{pmatrix} 
x_{2}x_{3}x_{4} & x_{2}x_{4}y_{3} \\ 0 & 0
\end{pmatrix}\begin{pmatrix} 
0 & 1 \\ 0 & 0
\end{pmatrix}) \\ 
= [\begin{pmatrix} 
0 & 1 \\ 0 & 0
\end{pmatrix},\begin{pmatrix} 
x_{1}x_{2}x_{3} & x_{1}x_{2}y_{3} \\ 0 & 0
\end{pmatrix}]\begin{pmatrix} 
x_{4} & y_{4} \\ 0 & 0
\end{pmatrix} \\
+\begin{pmatrix} 
x_{1} & y_{1} \\ 0 & 0
\end{pmatrix}[\begin{pmatrix} 
0 & 1 \\ 0 & 0
\end{pmatrix},\begin{pmatrix} 
x_{2}x_{3}x_{4} & x_{2}x_{4}y_{3} \\ 0 & 0
\end{pmatrix}] \\ 
= [\begin{pmatrix} 
0 & 1 \\ 0 & 0
\end{pmatrix},\begin{pmatrix} 
x_{1} & y_{1} \\ 0 & 0
\end{pmatrix}\begin{pmatrix} 
x_{2} & y_{2} \\ 0 & 0
\end{pmatrix}\begin{pmatrix} 
x_{3} & y_{3} \\ 0 & 0
\end{pmatrix}]\begin{pmatrix} 
x_{4} & y_{4} \\ 0 & 0
\end{pmatrix} \\
+\begin{pmatrix} 
x_{1} & y_{1} \\ 0 & 0
\end{pmatrix}[\begin{pmatrix} 
0 & 1 \\ 0 & 0
\end{pmatrix},\begin{pmatrix} 
x_{4} & y_{4} \\ 0 & 0
\end{pmatrix}\begin{pmatrix} 
x_{2} & y_{2} \\ 0 & 0
\end{pmatrix}\begin{pmatrix} 
x_{3} & y_{3} \\ 0 & 0
\end{pmatrix}] \\
=D(\begin{pmatrix} 
x_{1} & y_{1} \\ 0 & 0
\end{pmatrix},\begin{pmatrix} 
x_{2} & y_{2} \\ 0 & 0
\end{pmatrix},\begin{pmatrix} 
x_{3} & y_{3} \\ 0 & 0
\end{pmatrix})\begin{pmatrix} 
x_{4} & y_{4} \\ 0 & 0
\end{pmatrix} \\
+\begin{pmatrix} 
x_{1} & y_{1} \\ 0 & 0
\end{pmatrix}D(\begin{pmatrix} 
x_{4} & y_{4} \\ 0 & 0
\end{pmatrix},\begin{pmatrix} 
x_{2} & y_{2} \\ 0 & 0
\end{pmatrix},\begin{pmatrix} 
x_{3} & y_{3} \\ 0 & 0
\end{pmatrix}) \\
=\pi(D(a,b,c),d)+\pi(a,D(d,b,c)).\\
\\ $
Similarly, we have $D(a,\pi (b,d),c)=\pi(D(a,b,c),d)+\pi(b,D(a,d,c))$ and  $D(a,b,$\\$\pi(c,d))=\pi(D(a,b,c),d)+\pi(c,D(a,b,d))$. Thus $D$ is tri-derivation.

\end{example}
Now, we provide a necessary and sufficient condition such that the fourth adjoint  $D^{****}$ of a tri-derivation $D:A\times A\times A\longrightarrow X$ is again a tri-derivation. For  the fourth adjoint $D^{****}$  of a tri-derivation $D:A\times A\times A\longrightarrow X$, we are faced with the case eight:
\begin{eqnarray*}
&&(case 1)\ \ \ \ \ D^{****}:(A^{**},\square)\times (A^{**},\square)\times (A^{**},\square) \longrightarrow X^{**},\\
&&(case 2)\ \ \ \ \ D^{****}:(A^{**},\lozenge)\times (A^{**},\square)\times (A^{**},\square) \longrightarrow X^{**},\\
&&(case 3)\ \ \ \ \ D^{****}:(A^{**},\square)\times (A^{**},\lozenge)\times (A^{**},\square) \longrightarrow X^{**},\\
&&(case 4)\ \ \ \ \ D^{****}:(A^{**},\square)\times (A^{**},\square)\times (A^{**},\lozenge) \longrightarrow X^{**},\\
&&(case 5)\ \ \ \ \ D^{****}:(A^{**},\lozenge)\times (A^{**},\lozenge)\times (A^{**},\square) \longrightarrow X^{**},\\
&&(case 6)\ \ \ \ \ D^{****}:(A^{**},\lozenge)\times (A^{**},\square)\times (A^{**},\lozenge) \longrightarrow X^{**},\\
&&(case 7)\ \ \ \ \ D^{****}:(A^{**},\square)\times (A^{**},\lozenge)\times (A^{**},\lozenge) \longrightarrow X^{**},\\
&&(case 8)\ \ \ \ \ D^{****}:(A^{**},\lozenge)\times (A^{**},\lozenge)\times (A^{**},\lozenge) \longrightarrow X^{**}.
\end{eqnarray*}
In the following, we prove the state of case 1. The remaining state are proved in the same way.
\begin{theorem}\label{4.3}
Let $(\pi_{1},X,\pi_2)$ be a  Banach $A-$module and $D:A\times A\times A\longrightarrow X$ be a tri-derivation. Then $D^{****}:(A^{**},\square)\times (A^{**},\square)\times (A^{**},\square) \longrightarrow X^{**}$ is a tri-derivation if and only if 
\begin{enumerate}
\item $\pi_{2}^{**r*}(D^{****}(A,A,A^{**}),X^{*})\subseteq A^{*}$,

\item $\pi_{2}^{****}(X^{*},D^{****}(A,A^{**},A^{**}))\subseteq A^{*}$,

\item $D^{****r*}(\pi_{1}^{****}(X^{*},A^{**}),A^{**},A^{**})\subseteq A^{*}$,

\item $D^{******}(A^{**},\pi_{1}^{****}(X^{*},A^{**}),A)\subseteq A^{*}$,

\item $D^{*******}(A^{**},A^{**},\pi_{1}^{****}(X^{*},A^{**}))\subseteq A^{*}$.
\end{enumerate}
\end{theorem}
\begin{proof}
Let $D:A\times A\times A\longrightarrow X$ be a tri-derivation and (1),(2),(3),(4),(5) holds.
If $\{a_{\alpha}\},\{b_{\beta}\},\{c_{\gamma} \}$ and $\{d_{\tau} \}$ are bounded nets in $A$ , converging in $w^{*}-$topology
to $a^{**},b^{**},c^{**}$ and $d^{**}\in A^{**}$ respectively, in this case using (2), we conclude that 
$w^{*}-\lim\limits_{\alpha}w^{*}-\lim\limits_{\tau}w^{*}-\lim\limits_{\beta}w^{*}-\lim\limits_{\gamma}\pi_{2}(D(a_{\alpha},b_{\beta},c_{\gamma}),d_{\tau})=\pi_{2}^{***}(D^{****}(a^{**},b^{**},c^{**}),d^{**})$.
Thus for every $x^{*}\in X^{*}$ we have 
\begin{eqnarray*}
&&\langle D^{****}(\pi^{***}(a^{**},d^{**}),b^{**},c^{**}),x^{*}\rangle\\
&&=\lim\limits_{\alpha}\lim\limits_{\tau}\lim\limits_{\beta}\lim\limits_{\gamma}\langle x^{*},D(\pi(a_{\alpha},d_{\tau}),b_{\beta},c_{\gamma})\rangle\\
&&=\lim\limits_{\alpha}\lim\limits_{\tau}\lim\limits_{\beta}\lim\limits_{\gamma}\langle x^{*},\pi_{2}(D(a_{\alpha},b_{\beta},c_{\gamma}),d_{\tau})+\pi_{1}(a_{\alpha},D(d_{\tau},b_{\beta},c_{\gamma}))\rangle\\
&&=\lim\limits_{\alpha}\lim\limits_{\tau}\lim\limits_{\beta}\lim\limits_{\gamma}\langle x^{*},\pi_{2}(D(a_{\alpha},b_{\beta},c_{\gamma}),d_{\tau})\rangle\\
&&+\lim\limits_{\alpha}\lim\limits_{\tau}\lim\limits_{\beta}\lim\limits_{\gamma}\langle x^{*},\pi_{1}(a_{\alpha},D(d_{\tau},b_{\beta},c_{\gamma}))\rangle\\
&&=\langle x^{*},\pi_{2}^{***}(D^{****}(a^{**},b^{**},c^{**}),d^{**})\rangle +\langle x^{*},\pi_{1}^{***}(a^{**},D^{****}(d^{**},b^{**},c^{**}))\rangle\\
&&=\langle \pi_{2}^{***}(D^{****}(a^{**},b^{**},c^{**}),d^{**})+\pi_{1}^{***}(a^{**},D^{****}(d^{**},b^{**},c^{**})),x^{*}\rangle.
\end{eqnarray*}
Therefore
\begin{eqnarray*}
&&D^{****}(\pi^{***}(a^{**},d^{**}),b^{**},c^{**})\nonumber\\
&&=\pi_{2}^{***}(D^{****}(a^{**},b^{**},c^{**}),d^{**})+\pi_{1}^{***}(a^{**},D^{****}(d^{**},b^{**},c^{**})).
\end{eqnarray*}
Applying (1) and (3)  respectively, we can deduce that
$w^{*}-\lim\limits_{\alpha}w^{*}-\lim\limits_{\beta}w^{*}-\lim\limits_{\tau}w^{*}-\lim\limits_{\gamma}\pi_{2}(D(a_{\alpha},b_{\beta},c_{\gamma}),d_{\tau})=\pi_{2}^{***}(D^{****}(a^{**},b^{**},c^{**}),d^{**})$ and 
$w^{*}-\lim\limits_{\alpha}w^{*}-\lim\limits_{\beta}w^{*}-\lim\limits_{\tau}w^{*}-\lim\limits_{\gamma}\pi_{1}(b_{\beta},D(a_{\alpha},d_{\tau},c_{\gamma}))=\pi_{1}^{***}(b^{**},D^{****}(a^{**},d^{**},\\c^{**})).$
So in similar way, we can deduce that
\begin{eqnarray*}
&&D^{****}(a^{**},\pi^{***}(b^{**},d^{**}),c^{**})\nonumber\\
&&=\pi_{2}^{***}(D^{****}(a^{**},b^{**},c^{**}),d^{**})+\pi_{1}^{***}(b^{**},D^{****}(a^{**},d^{**},c^{**})).
\end{eqnarray*}
Applying (4) and (5), we can write
$w^{*}-\lim\limits_{\alpha}w^{*}-\lim\limits_{\beta}w^{*}-\lim\limits_{\gamma}w^{*}-\lim\limits_{\tau}\pi_{1}(c_{\gamma},D(a_{\alpha}\\,b_{\beta},d_{\tau}))=\pi_{1}^{***}(c^{**},D^{****}(a^{**},b^{**},d^{**})).$
Thus 
\begin{eqnarray*}
&&D^{****}(a^{**},b^{**},\pi^{***}(c^{**},d^{**}))\nonumber\\
&&=\pi_{2}^{***}(D^{****}(a^{**},b^{**},c^{**}),d^{**})+\pi_{1}^{***}(c^{**},D^{****}(a^{**},b^{**},d^{**})).
\end{eqnarray*}
By comparing equations (4.1), (4.2) and (4.3) follows that $D^{****}:(A^{**},\square)\times (A^{**},\square)\times (A^{**},\square) \longrightarrow X^{**}$ is a tri-derivation.

For the converse, let $D$ and $D^{****}:(A^{**},\square)\times (A^{**},\square)\times (A^{**},\square) \longrightarrow X^{**}$ be tri-derivation. We have to show that (1), (2), (3), (4) and (5) hold.  We shall only prove (2) the others parts  have similar argument. Fourth adjoint $D^{****}$ is  tri-derivation, thus we have
\begin{eqnarray*}
D^{****}(\pi ^{***}(a,d^{**}),b^{**},c^{**})&=&\pi_{2}^{***}(D^{****}(a,b^{**},c^{**}),d^{**})\\
&+&\pi_{1}^{***}(a,D^{****}(d^{**},b^{**},c^{**})).
\end{eqnarray*}
In the other hands, the mapping $D$ is tri-derivation, which follows that
\begin{eqnarray*}
D^{****}(\pi ^{***}(a,d^{**}),b^{**},c^{**})&=&w^{*}-\lim\limits_{\tau}w^{*}-\lim\limits_{\beta}w^{*}-\lim\limits_{\gamma}\pi_{2}(D(a,b_{\beta},c_{\gamma}),d_{\tau})\\
&+&\pi_{1}^{***}(a,D^{****}(d^{**},b^{**},c^{**})).
\end{eqnarray*}
Therefore follows that
\begin{eqnarray*}
&& \pi_{2}^{***}(D^{****}(a,b^{**},c^{**}),d^{**})\\&&=w^{*}-\lim\limits_{\tau}w^{*}-\lim\limits_{\beta}w^{*}-\lim\limits_{\gamma}\pi_{2}(D(a,b_{\beta},c_{\gamma}),d_{\tau}).
\end{eqnarray*}
 So, for every $d^{**}\in A^{**}$ we have
\begin{eqnarray*}
&&\langle \pi_{2}^{****}(x^{*},D^{****}(a,b^{**},c^{**})),d^{**}\rangle=\langle x^{*},\pi_{2}^{***}(D^{****}(a,b^{**},c^{**}),d^{**})\rangle\\
&&=\lim\limits_{\tau}\lim\limits_{\beta}\lim\limits_{\gamma}\langle x^{*},\pi_{2}(D(a,b_{\beta},c_{\gamma}),d_{\tau})\rangle=\lim\limits_{\tau}\lim\limits_{\beta}\lim\limits_{\gamma}\langle x^{*},\pi_{2}^{r}(d_{\tau},D(a,b_{\beta},c_{\gamma}))\rangle\\
&&=\lim\limits_{\tau}\lim\limits_{\beta}\lim\limits_{\gamma}\langle \pi_{2}^{r*}(x^{*},d_{\tau}),D(a,b_{\beta},c_{\gamma})\rangle=\lim\limits_{\tau}\lim\limits_{\beta}\lim\limits_{\gamma}\langle D^{*}(\pi_{2}^{r*}(x^{*},d_{\tau}),a,b_{\beta}),c_{\gamma}\rangle\\
&&=\lim\limits_{\tau}\lim\limits_{\beta}\langle c^{**},D^{*}(\pi_{2}^{r*}(x^{*},d_{\tau}),a,b_{\beta})\rangle=\lim\limits_{\tau}\lim\limits_{\beta}\langle D^{**}(c^{**},\pi_{2}^{r*}(x^{*},d_{\tau}),a),b_{\beta}\rangle\\
&&=\lim\limits_{\tau}\langle b^{**},D^{**}(c^{**},\pi_{2}^{r*}(x^{*},d_{\tau}),a)\rangle=\lim\limits_{\tau}\langle D^{***}(b^{**},c^{**},\pi_{2}^{r*}(x^{*},d_{\tau})),a\rangle\\
&&=\lim\limits_{\tau}\langle D^{****}(a,b^{**},c^{**}),\pi_{2}^{r*}(x^{*},d_{\tau})\rangle=\lim\limits_{\tau}\langle D^{****}(a,b^{**},c^{**}),\pi_{2}^{r*r}(d_{\tau},x^{*})\rangle\\
&&=\lim\limits_{\tau}\langle \pi_{2}^{r*r*}(D^{****}(a,b^{**},c^{**}),d_{\tau}),x^{*}\rangle=\lim\limits_{\tau}\langle \pi_{2}^{r*r**}(x^{*},D^{****}(a,b^{**},c^{**})),d_{\tau}\rangle\\
&&=\langle \pi_{2}^{r*r**}(x^{*},D^{****}(a,b^{**},c^{**})),d^{**}\rangle.
\end{eqnarray*}
As $\pi_{2}^{r*r**}(x^{*},D^{****}(a,b^{**},c^{**}))$ always lies in $A^{*}$, we have reached (2).
\end{proof}

For case 2, fourth adjoint $D^{****}$ of tri-derivation $D:A\times A\times A\longrightarrow X$ is a tri-derivation if and only if 
\begin{enumerate}
\item $\pi_{2}^{**r*}(D^{****}(A^{**},A^{**},A^{**}),X^{*})\subseteq A^{*}$,

\item $D^{****r*}(\pi_{1}^{****}(X^{*},A^{**}),A^{**},A^{**})\subseteq A^{*}$,

\item $D^{******}(A^{**},\pi_{1}^{****}(X^{*},A^{**}),A)\subseteq A^{*}$,

\item $D^{*******}(A^{**},A^{**},\pi_{1}^{****}(X^{*},A^{**}))\subseteq A^{*}$.
\end{enumerate}
For case 3, fourth adjoint $D^{****}$ of tri-derivation $D:A\times A\times A\longrightarrow X$ is a tri-derivation if and only if 
\begin{enumerate}
\item $\pi_{2}^{****}(X^{*},D^{****}(A,A^{**},A^{**}))\subseteq A^{*}$,

\item $D^{******}(A^{**},\pi_{1}^{****}(X^{*},A^{**}),A)\subseteq A^{*}$,

\item $D^{*******}(A^{**},A^{**},\pi_{1}^{****}(X^{*},A^{**}))\subseteq A^{*}$.
\end{enumerate}
For case 4, fourth adjoint $D^{****}$ of tri-derivation $D:A\times A\times A\longrightarrow X$ is a tri-derivation if and only if 
\begin{enumerate}
\item $\pi_{2}^{**r*}(D^{****}(A,A,A^{**}),X^{*})\subseteq A^{*}$,

\item $\pi_{2}^{****}(X^{*},D^{****}(A,A^{**},A^{**}))\subseteq A^{*}$,

\item $D^{****r*}(\pi_{1}^{****}(X^{*},A^{**}),A^{**},A^{**})\subseteq A^{*}$,

\item $D^{*****}(\pi_{1}^{****}(X^{*},A^{**}),A,A)\subseteq A^{*}$,

\item $D^{******}(A^{**},\pi_{1}^{****}(X^{*},A^{**}),A)\subseteq A^{*}$,

\item $D^{*******}(A^{**},A^{**},\pi_{1}^{****}(X^{*},A^{**}))\subseteq A^{*}$.
\end{enumerate}
For case 5, fourth adjoint $D^{****}$ of tri-derivation $D:A\times A\times A\longrightarrow X$ is a tri-derivation if and only if 
\begin{enumerate}
\item $\pi_{2}^{**r*}(D^{****}(A^{**},A^{**},A^{**}),X^{*})\subseteq A^{*}$,

\item $\pi_{2}^{****}(X^{*},D^{****}(A,A^{**},A^{**}))\subseteq A^{*}$,

\item $D^{****r*}(\pi_{1}^{****}(X^{*},A^{**}),A^{**},A^{**})\subseteq A^{*}$,

\item $D^{******}(A^{**},\pi_{1}^{****}(X^{*},A^{**}),A)\subseteq A^{*}$,

\item $D^{*******}(A^{**},A^{**},\pi_{1}^{****}(X^{*},A^{**}))\subseteq A^{*}$.
\end{enumerate}
For case 6, fourth adjoint $D^{****}$ of tri-derivation $D:A\times A\times A\longrightarrow X$ is a tri-derivation if and only if
\begin{enumerate}
\item $\pi_{2}^{**r*}(D^{****}(A^{**},A^{**},A^{**}),X^{*})\subseteq A^{*}$,

\item $D^{****r*}(\pi_{1}^{****}(X^{*},A^{**}),A^{**},A^{**})\subseteq A^{*}$,

\item $D^{*****}(\pi_{1}^{****}(X^{*},A^{**}),A,A)\subseteq A^{*}$,

\item $D^{******}(A^{**},\pi_{1}^{****}(X^{*},A^{**}),A)\subseteq A^{*}$,

\item $D^{*******}(A^{**},A^{**},\pi_{1}^{****}(X^{*},A^{**}))\subseteq A^{*}$.
\end{enumerate}
For case 7, fourth adjoint $D^{****}$ of tri-derivation $D:A\times A\times A\longrightarrow X$ is a tri-derivation if and only if
\begin{enumerate}
\item $\pi_{2}^{****}(X^{*},D^{****}(A,A^{**},A^{**}))\subseteq A^{*}$,

\item $\pi_{2}^{**r*}(D^{****}(A,A,A^{**}),X^{*})\subseteq A^{*}$,

\item $D^{*****}(\pi_{1}^{****}(X^{*},A^{**}),A,A)\subseteq A^{*}$,

\item $D^{******}(A^{**},\pi_{1}^{****}(X^{*},A^{**}),A)\subseteq A^{*}$,

\item $D^{*******}(A^{**},A^{**},\pi_{1}^{****}(X^{*},A^{**}))\subseteq A^{*}$.
\end{enumerate}
For case 8, fourth adjoint $D^{****}$ of tri-derivation $D:A\times A\times A\longrightarrow X$ is a tri-derivation if and only if
\begin{enumerate}
\item $\pi_{2}^{****}(X^{*},D^{****}(A,A^{**},A^{**}))\subseteq A^{*}$,

\item $\pi_{2}^{**r*}(D^{****}(A^{**},A^{**},A^{**}),X^{*})\subseteq A^{*}$,

\item $D^{****r*}(\pi_{1}^{****}(X^{*},A^{**}),A^{**},A^{**})\subseteq A^{*}$,

\item $D^{*****}(\pi_{1}^{****}(X^{*},A^{**}),A,A)\subseteq A^{*}$,

\item $D^{******}(A^{**},\pi_{1}^{****}(X^{*},A^{**}),A)\subseteq A^{*}$,

\item $D^{*******}(A^{**},A^{**},\pi_{1}^{****}(X^{*},A^{**}))\subseteq A^{*}$.
\end{enumerate}
\begin{remark}
For adjoint $D^{r****r}$ of tri-derivation $D:A\times A\times A\longrightarrow X$ we have the same argument.
\end{remark}


%
%



\end{document}